\documentclass[12pt]{article}

\usepackage{amsfonts}
\usepackage{amssymb}
\usepackage{amsthm}
\usepackage{amsmath}
\usepackage{mathrsfs}
\usepackage{xcolor}
\usepackage[all]{xy}
\usepackage[utf8]{inputenc}

\def\un{\mathbf{1}}
\def\mpn{\medskip\par\noindent}

\newcommand{\gMod}[1]{#1{\hbox{-}\mathsf{Mod}}}
\def\zero{\{0\}}

\usepackage[left=1.20in,right=1.20in]{geometry}
\definecolor{ao(english)}{rgb}{0.0, 0.5, 0.0}
\definecolor{brickred}{rgb}{0.8, 0.25, 0.33}
\definecolor{burntorange}{rgb}{0.8, 0.33, 0.0}
\definecolor{beaver}{rgb}{0.62, 0.51, 0.44}
\definecolor{brown(traditional)}{rgb}{0.59, 0.29, 0.0}
\definecolor{ao(english)}{rgb}{0.0, 0.5, 0.0}
\definecolor{verde}{rgb}{0.12, 0.8, 0.17}
\def\rouge{\color{red}}

\def\sur{\overline}
\newcommand{\rac}{\mathbb Q}

\newcommand{\cp}{\mathcal P}

\newcommand{\Ind}{\mathrm{Ind}}
\newcommand{\Res}{\mathrm{Res}}
\newcommand{\Def}{\mathrm{Def}}
\newcommand{\Inf}{\mathrm{Inf}}
\newcommand{\Iso}{\mathrm{Iso}}
\newcommand{\Hom}{\mathrm{Hom}}
\newcommand{\Aut}{\mathrm{Aut}}

\newcommand{\Ker}{\mathrm{Ker}}

\theoremstyle{plain}
\newtheorem{teo}{Theorem}
\newtheorem*{teo-non}{Theorem}
\newtheorem{prop}[teo]{Proposition}

\newtheorem{lema}[teo]{Lemma}

\theoremstyle{definition}
\newtheorem{defi}[teo]{Definition}
\newtheorem{nota}[teo]{Notation}

\theoremstyle{remark}

\newtheorem{rem}[teo]{Remark}

\def\CD{\mathcal{D}}
\def\CP{\mathcal{P}}

\def\C{\mathbb{C}}
\def\F{\mathbb{F}}
\def\Q{\mathbb{Q}}
\def\Z{\mathbb{Z}}
\def\mpn{\medskip\par\noindent}

\def\op{^{\mathrm{op}}}
\def\endpf{~\leaders\hbox to 1em{\hss\  \hss}\hfill~\raisebox{.5ex}{\framebox[1ex]{}}\smallskip\par}

\author{Serge Bouc and Nadia Romero\\
}
\title{Shifted functors of linear representations are semisimple}
\date{}

\begin{document}

\maketitle
\begin{minipage}{14cm}
\begin{footnotesize}
{\bf Abstract:} We prove that,  for any fields $k$ and $\mathbb{F}$ of characteristic $0$ and any finite group $T$, the category of modules over the shifted Green biset functor $(kR_{\mathbb{F}})_T$ is semisimple.\vspace{1ex}
\end{footnotesize}
\end{minipage}

\section{Introduction}
Shifted biset functors where introduced by the first author in 2010 (see Section 8.2 of \cite{biset}). The shifting of a functor $F$ by a group $T$ is also called the Yoneda-Dress construction at $T$ of the functor $F$. For a group $T$, this construction defines a self-adjoint exact linear endofunctor of the category of biset functors (Proposition 8.2.7 in~\cite{biset}). It is also known that a shifted Green biset functor is again a Green biset functor (Lemma 4.4 in \cite{lachica}) and that, for a Green biset functor $A$, a shifted $A$-module is an $A$-module (Example 8.5.6 in \cite{biset}). Shifted functors play a crucial role in the definition of the internal hom and the tensor product of biset functors. Furthermore, they have shown to be of major importance in the description of the category of modules over a Green biset functor.

In his PhD thesis~\cite{tesbenja},  B. Garc\'ia conjectured that, for any fields $k$ and $\mathbb{F}$ of characteristic~$0$ and any finite group $T$, the category of modules over the shifted Green biset functor $(kR_{\mathbb{F}})_T$ is semisimple (see Conclusiones in \cite{tesbenja}). Particular cases of this conjecture have already been proved; L. Barker showed in \cite{barker} that the category of $kR_\Q$-modules is semisimple, and Garc\'ia showed in his thesis (see also \cite{garcia-shifted}) that the same holds for $(\mathbb{C}R_{\mathbb{C}})_T$-modules, for any finite group $T$.

 In this article we prove Garc\'ia's conjecture, it is Theorem \ref{shifted semisimple} in Section \ref{shifted}. Then, we determine the multiplicity of each simple module in any $(kR_{\mathbb{F}})_T$-module in the case where $(kR_{\mathbb{F}})_T$ is restricted to groups $G$ such that $(|T|,\, |G|)=1$. The last result generalizes Theorem 1.9 in \cite{barker}. To prove the conjecture, we introduce in Section \ref{criterion}, a criterion to determine when a Green biset functor $A$ is semisimple (i.e. when every $A$-module is semisimple). This criterion will be used again in a forthcoming article, dealing with \textit{Green fields}.

\section{Preliminaries}

%We assume that the reader is familiar with the definitions of Green biset functors and their modules. %over a Green biset functor. Examples and some properties of Green biset functors and their modules can be found in sections 1.1  and 1.2 of \cite{centros}. As in that article, 
We consider the biset category defined over a commutative unital ring $k$ and having as objects a class $\mathcal{D}$ of finite groups, closed under subquotients and direct products. Unless otherwise stated, when we write \textit{Green biset functor} we will mean   \textit{Green $\CD$-biset functor over $k$}, for some $\mathcal{D}$ and $k$.  The Burnside functor will be denoted by $kB$, as usual, and if $\mathbb{F}$ is a field of characteristic $0$, the functor of linear $\mathbb{F}$-representations will be denoted by $kR_\mathbb{F}$.

We recall the definition of the category $\mathcal{P}_A$, associated to the Green biset functor~$A$.

The trivial group will be denoted by $\un$.  For a finite group $G$, we denote by $\Delta(G)$ the diagonal subgroup of $G\times G$, that is $\Delta(G)=\{(g,g)\mid g\in G\}$.

\begin{defi} \label{PA}Let $A$ be a Green  biset functor with identity element $\varepsilon\in A(\un)$. The category $\cp_A$ is defined in the following way:
\begin{itemize}
 \item The objects of $\cp_A$ are all finite groups {in $\CD$}.
 \item If $G$ and $H$ are groups {in $\CD$}, then ${\Hom}_{\cp_A}(H,\, G)=A(G\times H)$.
 \item Let $H,\, G$ and $K$ be groups {in $\CD$}. The composition of $\beta\in A(H\times G)$ and $\alpha\in A(G\times K)$ in $\cp_A$ is the following:
\begin{displaymath}
\beta \circ \alpha = A\left(\Def^{\,H\times\Delta(G)\times K}_{H\times K}\circ\Res^{H\times G\times G\times K}_{H\times\Delta(G)\times K}\right)(\beta\times\alpha).
\end{displaymath}
\item For a group $G$ in $\CD$, the identity morphism $\varepsilon_G$ of $G$ in $\cp_A$ is $A(\Ind_{\Delta(G)}^{G\times G}\circ\nolinebreak \Inf_\un^{\,\Delta (G)})(\varepsilon)$.
\end{itemize}
\end{defi}

For a morphism $\alpha\in A(H\times G)$ from $G$ to $H$ in the category $\CP_A$, we denote by $\alpha\op$ the element of $A(G\times H)$ - i.e. the morphism from $H$ to $G$ - defined by
$$\alpha\op=A(\Iso_{H\times G}^{G\times H})(\alpha),$$
where $\Iso_{H\times G}^{G\times H}$ is the group isomorphism $H\times G\to G\times H$ swapping the components.

For a Green biset functor $A$, an $A$-module is defined as a biset functor $M$, together with  bilinear products $A(G)\times M(H)\rightarrow M(G\times H)$ that satisfy natural conditions of associativity, identity element and functoriality. It is well known that this definition is equivalent to defining an $A$-module  as  a $k$-linear functor from the category $\cp_A$ to $\gMod{k}$. We denote the category of $A$-modules by  $\gMod{A}$. Important objects in $\gMod{A}$ are the shifted functors  (also called  shifted modules) $M_L$. If $L$ is a fixed finite group, the functor $M_L$ is defined as $M(G\times L)$ in a group $G\in \mathcal{D}$ and as $M(\alpha \times L)$ in an arrow $\alpha\in A(H\times G)$, for more details see Definition 10 in \cite{centros}. 

Given a Green biset functor $A$, defined on a class $\mathcal{D}$, we can generalize the Yoneda-Dress construction of $A$ in the following way. Let $T$ be a group in $\mathcal{D}$ and $\mathcal{D}'$ be a subclass of $\mathcal{D}$  possibly not containing $T$. The shifted functor $A_T$ sending $G\in \mathcal{D}'$ to $A(G\times T)$ and an element $\varphi\in kB(H,\, G)$ to $A(\varphi\times T)$ is a Green $\mathcal{D}'$-biset functor. 

When dealing with the simple objects of $\gMod{A}$, we will need the following notions.

\begin{defi}
Let $A$ be a Green biset functor. For a group $H\in \mathcal{D}$, the \textit{essential algebra}, $\widehat{A}(H)$, of $A$ on $H$, is the quotient of $A(H\times H)$ over the ideal generated by elements of the form $a\circ b$, where $a$ is in $A(H\times K)$, $b$ is in $A(K\times H)$ and $K$ runs over the groups in $\mathcal{D}$ of order smaller than $|H|$.
\end{defi}

The following functors can  be defined in more general settings, we recall the definitions in the context of $A$-modules. 

\begin{defi}
Let $A$ be a Green biset functor, $H$ a group in $\mathcal{D}$ and $V$ an $A(H\times H)$-module, the $A$-module $L_{H,\, V}$ is defined in $G\in \mathcal{D}$ as
\begin{displaymath}
L_{H,\, V}(G)=A(G\times H)\otimes_{A(H\times H)}V,
\end{displaymath}
and in an obvious way in arrows.
\end{defi}

If $V$ is a simple $A(H\times H)$-module, then  $L_{H,\,V}$ has a unique maximal subfunctor $J_{H,\,V}$, which in each evaluation is equal to 
\begin{displaymath}
J_{H,\, V}(G)=\left\{\sum_{i=1}^na_i\otimes v_i\in L_{H,\, V}(G)\mid\sum_{i=1}^n(b\circ a_i) v_i=0\ \forall b\in A(H\times G) \right\}.
\end{displaymath}

The quotient $S_{H,\, V}= L_{H,\,V}/J_{H,\,V}$ is a simple $A$-module such that $S_{H,\, V}(H)=V$. On the other hand, if $S$ is a simple $A$-module and $H\in \mathcal{D}$ is such that $S(H)\neq \zero$, then  taking $V=S(H)$ gives $S\cong S_{H,\,V}$.

%\begin{obs}
%{\color{red} quitar?}$A(1)\neq 0$.
%\end{obs}
%The map $A(G)\otimes_{A(1)} M(1)\to M(G)$ is induced by the product $\times$ on $M$. {\color{red} quitar?}
The following lemma is a generalization of Lemma 3.3. in \cite{lachica}, its proof is an easy modification of that of Lemma 3.3 in \cite{lachica}, so we give just a sketch of it, as a reminder.

\begin{lema}
\label{simquo}
Let $A$ and $C$ be Green biset functors and suppose there exists an epimorphism $f:A\rightarrow C$ of Green biset functors. Then a simple $C$-module is a simple $A$-module $S$  for which there exists a finite group $H\in\CD$ such that $S(H)\neq \zero$ and $f_{H\times H}$ induces a $C(H\times H)$-module structure on $S(H)$.
\end{lema}
\begin{proof}
It is easy to see that the simple $C$-modules coincide with the simple $A$-modules in which $\Ker f$ has a zero action. Hence, clearly if $S$ is such an $A$-module, then $f_{H\times H}$ induces an action of $C(H\times H)$ on $S(H)$ for any fixed $H$ such that $S(H)\neq \zero$. It remains then to see that if a simple $A$-module $S$ satisfies this last condition, then $\Ker f$ has a zero action on $S$. To do this we take arbitrary groups $K$ and $G$, we have that $S(G)$ is isomorphic to $L_{H,\, V}(G)$ over $J_{H,\, V}(G)$. Given $a\in \Ker f(K)$, we have that $a\times\varphi$ is in $\Ker f(K\times G\times H)$ for any $\varphi\in A(G\times H)$, and so $\psi\circ(a\times \varphi)$ is in $\Ker f(H\times H)$ for any $\psi\in A(H\times K\times G)$. This allows us to show that $\Ker f(K)$ sends $L_{H,\, V}(G)$ to $J_{H,\, V}(K\times G)$, hence it sends $S(G)$ to zero in $S(K\times G)$.
\end{proof}

\section{A semisimplicity criterion}
\label{criterion}

\begin{defi}  An $A$-module is called {\em semisimple} if it is the sum of its simple $A$-submodules.
A Green biset functor $A$ is called {\em semisimple}  if all $A$-modules are semisimple.
\end{defi}
\begin{lema}\label{lemma-semisimple}
A Green biset functor $A$ is semisimple if and only if for every group $H$, the $A$-module $A_H$ is  semisimple. 
\end{lema}
\begin{proof} If $A$ is semisimple, then in particular every $A$-module of the form $A_H$ is semisimple. Conversely, suppose that every $A_H$ is semisimple. As these modules can be view as the representable functors of the category $\mathcal{P}_A$, any $A$-module is a quotient of a direct sum of functors of the form $A_{H_i}$. Hence any $A$-module is semisimple.
\end{proof}

\begin{prop}\label{implies semisimple}
Let $A$ be a Green biset functor over a field $k$. Assume the following:
\begin{enumerate}
\item If $L$ is a finite group  in $\CD$, the $k$-algebra  $A(L\times L)$ is semisimple.
\item If $H$ is a finite group  in $\CD$, the bilinear map
$$(\alpha,\beta)\in A(H\times L)^2\mapsto \beta^{\rm op}\circ\alpha\in A(L\times L)$$
is non degenerate, i.e. if $\beta\op\circ\alpha=0$ for all $\beta\in A(H\times L)$, then $\alpha=0$.
\end{enumerate}
Then $A$ is a semisimple Green biset functor. 
\end{prop}
\begin{proof}
We will show that $A_L$ is a semisimple $A$-module, for each finite group $L$. Let $M$ be an $A$-submodule of $A_L$. 
\begin{itemize}
\item First, $M(L)$ is a left ideal of the algebra $A(L\times L)$, and we claim that if $M(L)=\zero$, then $M=0$. Indeed, if $M(L)=\zero$, then for any finite group $H$, the $k$-vector space $M(H)$ is contained in the set of elements $\alpha$ of $A_L(H)=A(H\times L)$ such that $\beta\op\circ\alpha=0$ for each $\beta\in A(H\times L)$ (since $\beta\op\circ \alpha\in M(L)=\zero$). By assumption 2, we have $M(H)=\zero$, as claimed. 
\item Now let $\Phi$ be the correspondence sending an $A$-submodule $M$ of $A_L$ to the (left) ideal $M(L)$ of the algebra $A(L\times L)$. In the other direction, let $\Psi$ be the correspondence sending a left ideal $V$ of $A(L\times L)$ to the $A$-submodule $\langle V\rangle$ of $A_L$ generated by $V$.\par
Clearly $\Phi\circ\Psi(V)=V$, since $\langle V\rangle(L)=A(L\times L)(V)=V$.\par
  Conversely, let $M$ be an $A$-submodule of $A_L$. Then $M(L)$ is a left ideal of $A(L\times L)$, which is semisimple by assumption 1. Then (see \cite{anderson-fuller} Theorem 9.6) there exists a left ideal  $W$ of $A(L\times L)$ such that
$$M(L)\oplus W=A(L\times L).$$
Let $P=M+\Psi(W)$. The evaluation of $P$ at $L$ is 
$$P(L)=M(L)+W=A(L\times L).$$
Then $P=A_L$, because $A_L$ is generated by $A_L(L)=A(L\times L)$.  Moreover, the intersection $I=M\cap\Psi(W)$, evaluated at $L$, is equal to
$$I(L)=M(L)\cap W=\zero.$$
It follows that $I=\zero$, and then $M\oplus \Psi(W)=A_L$.  \par
Now consider $M'=\Psi\circ\Phi(M)$. This is an $A$-submodule of $M$, and the same arguments show that $M'\oplus \Psi(W)=A_L$ also. But we have
$$M'\leq M\leq M'\oplus \Psi(W),$$
and then $M=M'\oplus\big(M\cap \Psi(W)\big)=M'\oplus I=M'$.\par
We have shown that $\Phi$ and $\Psi$ are mutually inverse bijections between the poset of $A$-submodules of $A_L$ and the poset of left ideals of $A(L\times L)$. Since the minimal elements of these posets are the simple $A$-submodules and the simple $A(L\times L)$-submodules, repectively, and since the least upper bound in each of these posets is the sum of submodules, it follows that $A_L$ is equal to the sum of its simple $A$-submodules. In other words $A_L$ is a semisimple $A$-module, for each finite group~$L$. By Lemma \ref{lemma-semisimple}, $A$ is semisimple.
\end{itemize}
\end{proof}

%\begin{rem}
%The above proof is formally correct without the assumption that $k$ be a field. But the notion of semisimple algebra over a commutative ring $k$ seems to make sense only with this assumption, or at least that $k$ be semisimple. 
%\end{rem}
\begin{rem}
One can show that the abovementioned Theorem 9.6 of \cite{anderson-fuller} can be generalized to modules over a Green biset functor $A$: An $A$-module $M$ is semisimple if and only if any $A$-submodule $N$ of $M$ admits a complement in $M$ (that is, an $A$-submodule $L$ of $M$ such that $N\oplus L=M$). Nevertheless, the proof of this fact is not straightforward. Notice that this is not used in the proof of Proposition \ref{implies semisimple}. 
\end{rem}

\section{Shifted functors of linear representations}
\label{shifted}
In this section $k$ denotes a field of characteristic 0.

\begin{teo} \label{shifted semisimple}Let  $\F$ be a field of characteristic 0. For any finite group $T$, the functor $(kR_\F)_T$ is a semisimple Green biset functor.
\end{teo}
\begin{proof}
 We will show that $(kR_\F)_T$ fulfills the assumptions 1 and 2 of Proposition \ref{implies semisimple}. 

First step: we want to show that the bilinear map
\begin{align*}
	A(H\times L)\otimes A(H\times L)&\to A(L\times L)\\
	\alpha\otimes\beta&\mapsto\xi_k(\alpha,\beta)=\beta\op\circ \alpha
\end{align*}
is non degenerate when $A=(kR_\F)_T$, for any finite groups $H$ and $L$, in the sense that if $\xi_k(\alpha,\beta)=0$ for all $\beta\in A(H\times L)$, then $\alpha=0$.\par
For given groups $H$ and $L$, we have two $k$-vector spaces
$$U_k=k\otimes R_\F(H\times L\times T),\;\; W_k=k\otimes R_\F(L\times L\times T)$$
Saying that $\xi_k$ is non degenerate amounts to saying that the map
\begin{align*}
	\widehat{\xi}_k:U_k&\to \Hom_k(U_k,W_k)\\
	u&\mapsto \big(v\mapsto \xi_k(u,v)\big)
\end{align*}
is injective.\par
Now each of the $k$-vector spaces $U_k$ and $W_k$ is obtained by scalar extension $k\otimes_\Z(-)$ from the corresponding groups $U_\Z=R_\F(H\times L\times T)$, and $W_\Z=R_\F(H\times L\times T)$. These
two groups are free abelian groups of finite rank. It follows that
\begin{align*}
	\Hom_k(U_k,W_k)&=\Hom_k(k\otimes_\Z U_\Z,k\otimes_\Z W_\Z)\\
	&\cong k\otimes_\Z\Hom_\Z(U_\Z,W_\Z),
\end{align*}
and that $\widehat{\xi}_k=k\otimes\widehat{\xi}_\Z$. If we can prove that $\widehat{\xi}_\Z$ is injective, then $\widehat{\xi}_k$ will be injective as well, because $k$ is flat over $\Z$.\par
Now in order to prove this, we will build a bilinear form
$$\langle -,-\rangle :R_\F(H\times L\times T)\times R_\F(H\times L\times T)\to \Z$$
with the following two properties:
\begin{enumerate}
\item The form $\langle -,-\rangle$ is positive definite, that is, $\langle \alpha,\alpha\rangle\geq 0$ for any element $\alpha\in R_\F(H\times L\times T)$, with equality if and only if $\alpha=0$.
\item The kernel of $\widehat{\xi}_\Z$, i.e. the set of elements $\alpha\in R_\F(H\times L\times T)$ such that $\beta\op\circ\alpha=0$ for all $\beta\in R_\F(H\times L\times T)$, is equal to the left radical of $\langle -,- \rangle$, that is, the set of elements $\alpha\in R_\F(H\times L\times T)$ such that $\langle\alpha,\beta\rangle=0$ for all $\beta\in R_\F(H\times L\times T)$.
\end{enumerate}
If we can do this, then the kernel of $\widehat{\xi}_\Z$ will be $\zero$, for $\langle -,- \rangle$ is positive definite, so its radical is zero.\medskip\par
We first introduce a linear form $\tau$ on $R_\F(L\times L\times T)$, with values in $\Z$, as follows: for a finite dimensional $\F (L\times L\times T)$-module $V$, we set
\begin{equation} \label{tau}
\tau(V)=\dim_\F V^{\Delta L\times T},
\end{equation}
i.e. the dimension of the space of fixed points on $V$ under the action of the group
$$\Delta L\times T=\{(l,l,t)\mid l\in L,\;t\in T\}.$$
If $\chi_V$ is the character of $V$, using for example (\cite{alperin-bell} Sec. 14, Lemma 11), adapted to an arbitrary field of characteristic 0, we have
$$\tau(V)=\frac{1}{|L||T|}\sum_{\substack{l\in L\\t\in T}}\chi_V(l,l,t).$$
Next, for a finite dimensional $\F(H\times L\times T)$-module $U$, we endow the dual vector space $U^\sharp=\Hom_\F(U,\F)$ with the structure of $\F(L\times H\times T)$-module defined by
\begin{equation}\label{sharp}
\forall (l,h,t)\in L\times H\times T,\;\forall\alpha\in U^\sharp,\;\forall u\in U,\;\;\big((l,h,t)\alpha\big)(u)=\alpha\big((h^{-1},l^{-1},t^{-1})u\big).
\end{equation}
Then the character $\chi_{U^\sharp}$ of $U^\sharp$ is given by
$$\forall (l,h,t)\in L\times H\times T,\;\;\chi_{U^\sharp}(l,h,t)=\chi_U(h^{-1},l^{-1},t^{-1}),$$
where $\chi_U$ is the character of $U$.\par
Now when $U$ and $V$ are finite dimensional $\F(H\times L\times T)$-modules, we set
\begin{equation}\label{bilinear}
\langle U,V\rangle =\tau(V^\sharp\circ U),
\end{equation}
and we extend this definition to a bilinear form $\langle -,-\rangle$ defined on $R_\F(H\times L\times T)$, with values in $\Z$. \par
By Lemma~7.1.3 of~\cite{biset}, the character $\chi_{V^\sharp\circ U}$ of the $\F(L\times L\times T)$-module $V^\sharp \circ U$ is given by
$$\forall l,l'\in L,\;t\in T,\;\;\chi_{V^\sharp\circ U}(l,l',t)=\frac{1}{|H|}\sum_{h\in H}\chi_V(h^{-1},l^{-1},t^{-1})\chi_U(h,l',t),$$
where $\chi_U$ and $\chi_V$ are the characters of $U$ and $V$, respectively. It follows that
$$\langle U,V\rangle=\frac{1}{|H||L||T|}\sum_{\substack{h\in H\\l\in L\\t\in T}}\chi_V(h^{-1},l^{-1},t^{-1})\chi_U(h,l,t).$$

We aim to show that this bilinear form is positive definite. To do this, we first choose an integer $n$ (e.g. $n=|H||L||T|$) such that all the character values of the groups we consider lie in the subfield $S$ of $\F$ generated by the $n^{th}$-roots of unity (that is, the intersection $\F\cap \Q[\zeta]$ in an algebraic closure of $\F$, where $\zeta$ is a primitive $n^{th}$-root of unity). We fix a field isomorphism $\Phi$ from $\Q[\zeta]$ to $\Q[s]$, where $s$ is a primitive $n^{th}$-root of unity in the field $\C$ of complex numbers. We can assume that $\Phi$ is the identity map on $\Q\subseteq \Q[\zeta]$.\par
With this notation, for $\F (H\times L\times T)$-modules $U$ and $V$, we have
$$\langle U,V\rangle=\Phi\big(\langle U,V\rangle\big)=\frac{1}{|H||L||T|}\sum_{\substack{x\in H\\y\in L\\t\in T}}\sur{\Phi\big(\chi_V(x,y,t)\big)}\Phi\big(\chi_U(x,y,t)\big).$$
Using this, we can consider the difference $U-V\in R_\F(H\times L\times T)$, and compute
\begin{align*}
	\langle U-V,U-V\rangle&=\Phi\big(\langle U,U\rangle\big)+\Phi\big(\langle V,V\rangle\big)-\Phi\big(\langle U,V\rangle\big)-\Phi\big(\langle V,U\rangle\big)\\
	&=\frac{1}{|H||L||T|}\sum_{\substack{x\in H\\y\in L\\t\in T}}(\sur{a_U^{x,y,t}}a_U^{x,y,t}+\sur{a_V^{x,y,t}}a_V^{x,y,t}-\sur{a_U^{x,y,t}}a_V^{x,y,t}-\sur{a_V^{x,y,t}}a_U^{x,y,t})\\
	&=\frac{1}{|H||L||T|}\sum_{\substack{x\in H\\y\in L\\t\in T}}\big|a_U^{x,y,t}-a_V^{x,y,t}\big|^2,
\end{align*}
where for short $a_U^{x,y,t}=\Phi\big(\chi_U(x,y,t)\big)$ and $a_V^{x,y,t}=\Phi\big(\chi_V(x,y,t)\big)$. \par
It follows that $\langle U-V,U-V\rangle$ is a non negative integer, equal to zero  if and only if $a_U^{x,y,t}=a_V^{x,y,t}$ for all $(x,y,t)\in H\times L\times T$, i.e. if $\chi_U=\chi_V$, that is, if $U$ and $V$ are isomorphic. In other words, the bilinear form $\langle-,-\rangle$ defined on  $R_\F(H\times L\times T)$, with values in $\Z$, is positive definite. \par
Now assume that $\alpha\in R_\F(H\times L\times T)$ is such that $\beta\op\circ\alpha=0$ for any $\beta\in R_\F(H\times L\times T)$. Then in particular 
$$\tau(\beta\op\circ\alpha)=0=\langle\alpha,\beta^*\rangle,$$
where $\beta\to\beta^*$ is the linear automorphism of $R_\F(H\times L\times T)$ induced by the duality $V\mapsto V^*=\Hom_\F(V,\F)$, with its usual structure of $\F(H\times L\times T)$-module. Taking $\beta=\alpha^*$, we get $\langle\alpha,\alpha\rangle=0$, hence $\alpha=0$, as was to be shown.\mpn

Second step: it remains to show that the $k$-algebra $A(L\times L)$ is semisimple for any finite group~$L$. Let us denote by $A_k(L\times L)$ the algebra $A(L\times L)$, in order to emphasize the field $k$. Then we have $A_k(L\times L)\cong k\otimes_\Q A_\Q(L\times L)$. If we can show that the $\Q$-algebra $A_\Q(L\times L)$ is semisimple, then $A_k(L\times L)$ will be semisimple as well, by Corollary~7.8 of~\cite{curtis}, as $k$ is a separable $\Q$-algebra since it has characteristic 0.\par
We use again the linear form $\tau:R_\F(L\times L\times T)\to \Z$ defined in~(\ref{tau}), and the associated bilinear form $\langle-,-\rangle$ defined on $R_\F(L\times L\times T)$ with values in $\Z$, defined in~(\ref{bilinear}), in the case $H=L$.
We observe moreover that in the case $H=L$, the construction $U\mapsto U^\sharp$ introduced in~(\ref{sharp}) extends to an involutive linear automorphism of $R_\F(L\times L\times T)$. \par 
Furthermore, when $U$ and $V$ are finite dimensional $\F(L\times L\times T)$-modules, the character of the composition $U\circ V$ is given by
$$\forall l,l'\in L,\;t\in T,\;\;\chi_{U\circ V}(l,l',t)=\frac{1}{|L|}\sum_{x\in L}\chi_U(l,x,t)\chi_V(x,l',t),$$
where $\chi_U$ and $\chi_V$ are the characters of $U$ and $V$, respectively. It follows easily that the $\F(L\times L\times L)$-modules $(U\circ V)^\sharp$ and $V^\sharp\circ U^\sharp$ have the same character, hence they are isomorphic. Thus, we can extend the construction $V\mapsto V^\sharp$ to an antiautomorphism of the algebra $A_\Q(L\times L)=\Q\otimes_\Z R_\F(L\times L\times T)$, that we denote by $\alpha\mapsto \alpha^\sharp$.\par

We also extend the bilinear form $\langle-,-\rangle$ to a bilinear form $\langle-,-\rangle_\Q$ on $A_\Q(L\times L)=\Q\otimes_\Z R_\F(L\times L\times T)$, with values in $\Q$. This form $\langle-,-\rangle_\Q$ is again positive definite: if $\alpha\in \Q\otimes_\Z R_\F(L\times L\times T)$, then there is an integer $m$ such that $m\alpha\in R_\F(L\times L\times T)$, and
$$\langle \alpha,\alpha\rangle_\Q=\frac{1}{m^2}\langle m\alpha, m\alpha\rangle$$
is a non negative rational number, equal to 0 if and only if $m\alpha=0$, that is if $\alpha=0$.\par
Now let $U$, $V$ and $W$ be finite dimensional $\F(L\times L\times T)$-modules. Then
\begin{align*}
	\langle U\circ V,W\rangle&=\tau\big(W^\sharp\circ(U\circ V)\big)\\
	&=\tau\big((W^\sharp\circ U)\circ V\big)\\
	&=\tau\big((U^\sharp\circ W)^\sharp\circ V\big)\\
	&=\langle V,U^\sharp\circ W\rangle.
\end{align*}
By linearity, we get that 
\begin{equation}\label{associative}
	\forall (\alpha,\beta,\gamma)\in A_\Q(L\times L)^3,\;\;	\langle \alpha\circ\beta,\gamma\rangle_\Q=\langle \beta,\alpha^\sharp\circ\gamma\rangle_\Q.
\end{equation}
Let $I$ be a left ideal of $A_\Q(L\times L)$. By equation~(\ref{associative}), the orthogonal $I^\perp$ of $I$ for $\langle-,-\rangle_\Q$ is also a left ideal of $A_\Q(L\times L)$. Moreover $I\cap I^\perp=\zero$ since $\langle-,-\rangle_\Q$ is positive. Finally $\dim_\Q I+\dim_\Q I^\perp=\dim_\Q A(L\times L)$ since $\langle-,-\rangle_\Q$ is non degenerate. It follows that $I\oplus I^\perp=A_\Q(L\times L)$, so any left ideal of $A_\Q(L\times L)$ has a complement left ideal in $A_\Q(L\times L)$. That is, the algebra $A_\Q(L\times L)$ is semisimple. This completes the proof of Theorem~\ref{shifted semisimple}.
\end{proof}

We devote the rest of the section to prove a generalization of Theorem 1.9 in \cite{barker}, regarding the multiplicities with which the simple modules appear in a $(kR_\rac)_T$-module, when $(kR_\rac)_T$ is defined in a class of groups $\mathcal{D}$ such that $(|T|,\, |G|)=1$ for every $G\in \mathcal{D}$.

For any group $T$, the linearization morphism defines an epimorphism of Green biset functors $kB_T\rightarrow (kR_{\rac})_T$ so, in view of Lemma \ref{simquo}, the simple $(kR_{\rac})_T$-modules are some of the simple $kB_T$-modules. It is shown in \cite{garcia-shifted} that any simple $(kR_{\rac})_T$-module has a unique (cyclic) minimal group. Hence, in this case we obtain an analogous result of Lemma 3.3 in \cite{lachica}: The simple $(kR_{\rac})_T$-modules are the simple $kB_T$-modules $S_{C,\, V}$ for $C$ a cyclic group and $V$ a simple $\widehat{kB_T}(C)$-module such that $\widehat{(kR_{\rac})_T}(C)\neq 0$ and $V$ is a (simple) $\widehat{(kR_{\rac})_T}(C)$-module. 

We recall the following notation from \cite{lachica}.

\begin{nota}
\label{notavieja}
Let $G$, $H$, $K$ and $T$ be groups. Let $E$ be a subgroup of $G\times H\times T$ and $D$ be a subgroup of $H\times K\times T$. Then $E* D$ denotes the following subgroup of $G\times K\times T$.
\begin{displaymath}
E*D=\{(g,\, k,\, t)\in G\times K\times T\mid \exists h\in H\textrm{ s.t. } (g,\, h,\, t)\in E\textrm{ and } (h,\, k,\, t)\in D \}.
\end{displaymath}
We also denote by $p_1(E)$, $p_2(E)$ and $p_3(E)$ the projections of $E$ on $G$, $H$ and $T$ respectively; $p_{1,\, 2}(E)$ will denote the projection over $G\times H$, and in the same way we define the other possible combinations of indices. We write $k_1(E)$ for 
\begin{displaymath}
\{h\in p_1(E)\mid (h,\, 1,\, 1)\in E\},
\end{displaymath}
which is a normal subgroup of $p_1(E)$. Similarly, we define $k_2(E)$ and $k_3(E)$.
\end{nota}

The following lemma is a straightforward generalization of the corresponding results appearing in Proposition 2.3.22 of \cite{biset}.
\begin{lema}
\label{pyk}
Let $E$ and $D$ be as in the previous notation. We have $p_1(E*D)\subseteq p_1(E)$ and $k_1(E)\subseteq k_1(E*D)$.
\end{lema}

In the remainder of this section we assume that $T$  is a group such $(|T|,\, |G|)=1$ for every $G\in \mathcal{D}$ and we consider the shifted functors $kB_T$ and $(kR_\rac)_T$ defined in $\mathcal{D}$.

\begin{lema}
\label{divide}
%Let $T$ be a finite group such that $(|T|,\, |G|)=1$ for all $G$ in $\mathcal{G}$. Consider the shifted functor $kB_T$ defined in $\mathcal{G}$. 
Let $S$ be a simple $kB_T$-module and $K$ be a minimal group for $S$. If $H$ is a group such that $S(H)\neq 0$, then $K$ is a isomorphic to a subquotient of $H$.  
\end{lema}
\begin{proof}
Since $V=S(K)\neq 0$,  we know that $S(H)$ is isomorphic to the quotient of $L_{K,\,V}(H)=kB_T(H\times K)\otimes_{kB_T(K\times K)}V$ by
\begin{displaymath}
J_{K,\, V}(H)=\left\{\sum_{i=1}^n\alpha_i\otimes v_i\in L_{K,\,V}(H)\mid \sum_{i=1}^n(\beta\circ\alpha_i)v_i=0\, \forall \beta\in kB_T(K\times H)\right\}.
\end{displaymath} 
Hence, if $S(H)\neq 0$, we must have $L_{K,\,V}(H)\neq J_{K,\, V}(H)$. This means that there exist $\alpha \in kB_T(H\times K)$ and $\beta \in kB_T(K\times H)$ such that $\beta\circ\alpha$ has a non zero action on $V$. In particular, $\beta\circ\alpha$ is not zero in $\widehat{kB_T}(K\times K)$. By Lemma 4.5 in \cite{lachica}, this means that there exist $E$ a subgroup of $K\times H\times T$ and $D$ a subgroup of $H\times K\times T$ such that
\begin{displaymath}
\left((K\times H\times T)/E\right)\circ\left((H\times K\times T)/D\right)=\sum_{(h,\, t)}(K\times K\times T)/E* ^{(h,\, 1,\,t)}D
\end{displaymath}
is not zero in $\widehat{kB_T}(K\times K)$, where $(h,\, t)$ runs through a set of representatives of the double cosets $p_{2,\, 3}(E)\backslash H\times T /p_{1,\, 3}(D)$. In turn, there exists $(h,\,t)\in H\times T$ such that $(K\times K\times T)/E* ^{(h,\, 1,\,t)}D$ is not zero in $\widehat{kB_T}(K\times K)$. By Lemma 4.8 in \cite{lachica}, this implies that $p_1(E* ^{(h,\, 1,\,t)}D)/k_1(E* ^{(h,\, 1,\,t)}D)$ is isomorphic to $K$. Then, by the previous lemma, $p_1(E)/k_1(E)$ is isomorphic to $K$. It follows that $E$ is equal to
\begin{displaymath}
E=\{(f(h,\, t),\, h,\, t)\mid (h,\, t)\in A\},
\end{displaymath}
where $A$ is a subgroup of $H\times T$ and $f$ is a surjective homomorphism from $A$ to $K$. Since $T$ has order relatively prime to $|H|$ and $|K|$, we have that $A=H_0\times T_0$, for some $H_0\leqslant H$ and $T_0\leqslant T$, and that  $1\times T_0$ is contained in $\Ker f$. Hence $K$ is isomorphic to a quotient of $H_0$, i.e. to a subquotient of $H$.
\end{proof}

\begin{rem}
Observe that, in the situation of the previous lemma, $K$ is the only minimal group of $S$ up to isomorphism. 
\end{rem}

It is shown in Section 3.3 of \cite{garcia-shifted} that  there is an isomorphism of $k$-algebras 
\begin{displaymath}
\widehat{(kR_\rac)_T}(G)\cong \prod_{\tau\in [CS(T)]}\widehat{kR_\rac}(G)
\end{displaymath}
where $[CS(T)]$ is a set of representatives of the conjugacy classes of cyclic subgroups of $T$. Hence, for  a simple $\widehat{(kR_\rac)_T}(G)$-module $W$, there exist a unique $\tau$ and a unique simple $\widehat{kR_\rac}(G)$-module $V$ such that $W$ is isomorphic to the restriction of scalars of $V$ via the projection on the $\tau$-th coordinate factor $\widehat{kR_\rac}(G)$. We will denote this module by $V_\tau$. Furthermore, by Corollary 3.24 of \cite{garcia-shifted}, the simple $(kR_\rac)_T$-modules are in correspondence with the triplets $(C,\, V,\, \tau)$ such that $C$ is a cyclic group, $V$ is a primitive $k\Aut(C)$-module and $\tau$ is cyclic subgroup of $T$. The group $C$ and the module $V$ are taken up to isomorphism and $\tau$ up to conjugation. We will abbreviate a triplet $(C,\, V,\, \tau)$ by $(C,\, V_\tau)$ and we will write $S_{C,\, V_\tau}$ for the corresponding simple $(kR_\rac)_T$-module.

%The following corollary generalizes Theorem 1.9 in \cite{barker}.

\begin{prop}
Let $A=(kR_\Q)_T$ be defined in $\mathcal{D}$, where $k$ is a field of characteristic~0 and $T$ is a finite group in $\mathcal{D}$ such that $(|T|,\, |G|)=1$ for all $G\in \mathcal{D}$. Given an $A$-module~$M$, then in the expression
\begin{displaymath}
M\cong \bigoplus_{(C,\, V_\tau)}m_{C,\, V_\tau}S_{C,\, V_\tau},
\end{displaymath}
for any  $(C, V_\tau)$, the multiplicity $m_{C,\, V_\tau}$ equals the multiplicity of $V_\tau$ in the \mbox{$kB_T(C\times C)$}-module $M(C)$.  In particular, for any cyclic group $C$ and any primitive $k\Aut(C)$-module~$V$, the multiplicities $m_{C,\, V_\tau}$ satisfy that
\begin{displaymath}
\sum_{\tau\in [CS(T)]}m_{C,\, V_\tau}
\end{displaymath}
is equal to the multiplicity of $V$ in the $k\Aut(C)$-module $M(C)$. 
\end{prop}
\begin{proof}
Suppose $M\cong \bigoplus_{(C,\, V_\tau)}m_{C,\, V_\tau}S_{C,\, V_\tau}$ and let $D$ be a cyclic group. By the paragraph before Notation \ref{notavieja} and  Lemma \ref{divide} we have
\begin{displaymath}
M(D)\cong \bigoplus_{\substack{(C,\, V_\tau)\\ |C|\textrm{ divides } |D|}}m_{C,\, V_\tau}S_{C,\, V_\tau}(D).
\end{displaymath}
Let $|D|=m$. We will show first that  no $ S_{C,\, V_\tau}(D)$ with $|C|$ a proper divisor of $m$  contains primitive $k\Aut(D)$-modules. %We recall that $M$ has a natural structure of $kB_T$-module and that the action of $k\Aut(D)$ in $M(D)$ can be given by the series of injective morphisms
%\begin{displaymath}
%k\Aut (D)\rightarrow kB(D\times D)\rightarrow kB_T(D\times D).
%\end{displaymath}
So,  if we suppose that $W_\rho=S_{D,\,W_\rho}(D)$ is contained in $S_{C,\,V_\tau}(D)$, for such a $C$, then $S_{C,\,V_\tau}(D)$ would contain a primitive $k\Aut(D)$-module, a contradiction.% We will show that this leads to a contradiction.\par}
%So, once we show that  $S_{C,\, V_\tau}(D)$ does not contain a primitive $k(\Z/m\Z)^{\times}$-module, we will also  have that $S_{D,\, W_\rho}(D)$ is not a $kB_T(D\times D)$-submodule of $S_{C,\, V_\tau}(D)$, for any other $S_{D,\, W_\rho}$.

%and $W$ be a primitive $k(\Z/m\Z)^{\times}$-module. %Suppose $m_{D,\, W}\neq 0$. 
%Clearly, $W$ is not contained in $S_{D,\, W'}(D)$ for any $W'$ not isomorphic to $W$. Hence, it suffices to show that $W$ is not contained in any of the evaluations $S_{C,\, V}(D)$ with $|C|<|D|$. 
%Take one such $C$ and let $|C|=n$.
We have that $S_{C,\, V_\tau}(D)$ is equal to a quotient of $L_{C,\, V_\tau}(D)=kB_T(D\times C)\otimes_{kB_T(C\times C)}V_\tau$. 
 Let us see that the kernel of the natural epimorphism
\begin{displaymath}
\pi_{m,\, n}:(\Z/m\Z)^{\times}\rightarrow(\Z/n\Z)^{\times},
\end{displaymath}
acts trivially in $L_{C,\, V_\tau}(D)$. Then, it acts trivially in $S_{C,\, V_\tau}(D)$. 
%by the submodule
%\begin{displaymath}
%J_{C,\, V}(D)=\left\{\sum_{i=1}^n\varphi_i\otimes v_i\mid \sum_{i=1}^n(\psi \varphi_i)\cdot v_i\forall \psi\in kB(C,\, D) \right\}.
%\end{displaymath}

For an element $r\in (\Z/m\Z)^{\times}$, we denote by $\Delta_r(D)$ the subgroup $\{(d,\, d^r)\mid d\in D\}$ of $D\times D$ and by $\beta_r$ the $(D\times D\times T)$-set $(D\times D\times T)/(\Delta_r(D)\times T)$. If $\alpha$ is a transitive $(D\times C\times T)$-set, say $\alpha=(D\times C\times T)/A$, then
\begin{displaymath}
\beta_r\circ \alpha\cong (D\times C\times T)/\big((\Delta_r(D)\times T)\ast A\big),
\end{displaymath}
by Lemma 4.5 in \cite{lachica}. Since $(|T|, \, |D|)=1$, we have that $A=A_0\times T_0$, where $A_0\leqslant D\times C$ and $T_0\leqslant T$, hence
\begin{displaymath}
\begin{array}{rcl}
(\Delta_r(D)\times T)\ast A&=&\{(d,\, c, \, t)\in D\times C\times T\mid (d^r,\, c,\, t)\in A\}\\
&=& \{(d,\, c)\in D\times C\mid (d^r,\, c)\in A_0\}\times T_0.
\end{array}
\end{displaymath}
But if $r$ is in $\Ker\,\pi_{m,\,n}$, then sending an element $(d,\, c)$ of $A_0$ to $(d^t,\, c)=(d,\, c)^t$ defines an automorphism on $A_0$, hence
\begin{displaymath}
\{(d,\, c)\in D\times C\mid (d^r,\, c)\in A_0\}=A_0.
\end{displaymath}
So, $(\Delta_r(D)\times T)\ast A=A$ and $\Ker\,\pi_{m,\,n}$ acts trivially in $L_{C,\, V_\tau}(D)$.

On the other hand, if $S_{D,\, Z_\gamma}$ is such that $\rho\neq \gamma$ in $[CS(T)]$ or $W\ncong Z$ as $k\Aut (D)$-modules, then $S_{D,\, W_\rho}(D)\ncong S_{D,\,Z_\gamma}(D)$ as $\widehat{(kR_\rac)_T}(D)$-modules. Now, we have the following chain of epimorphisms of rings
\begin{displaymath}
kB_T(D\times D)\rightarrow (kR_\rac)_T(D\times D)\rightarrow \widehat{(kR_\rac)_T}(D).
\end{displaymath}
So $S_{D,\, W_\rho}(D)\ncong S_{D,\,Z_\gamma}(D)$ in the $kB_T(D\times D)$-module $M(D)$. This proves the first statement in the proposition. The second statement comes from the first part of the proof and the fact that if we restrict the action from $kB_T(D\times D)$ to $k\Aut(D)$, then $S_{D,\, W_\tau}(D)\cong S_{D,\, W_\rho}(D)$, as $k\Aut(D)$-modules. 
\end{proof}

\begin{rem}
Let  $A$ be as in the previous proposition. If $S_{C,\, V_\tau}$, with $C$ a cyclic group and $V$ a primitive $k\Aut(C)$-module, is an $A$-submodule of $A_L$, then $S_{C,\, V_\tau}(L)\neq 0$ and so $C$ is a subquotient of $L$. That is, the only simple $A$-modules that may appear in the semisimple decomposition of $A_L$ are those corresponding to cyclic subquotients of $L$.
\end{rem}

%\bibliographystyle{abbrv}
%\bibliography{semi2}

\begin{thebibliography}{1}

\bibitem{alperin-bell}
J.~L. Alperin and R.~B. Bell.
\newblock {\em Groups and representations}, volume 162 of {\em Graduate Texts
  in Mathematics}.
\newblock Springer-Verlag, New York, 1995.

\bibitem{anderson-fuller}
F. W. Anderson and K. R. Fuller.
\newblock {\em Rings and categories of modules}, volume 13 of {\em Graduate Texts in Mathematics}, second edition.
\newblock Springer-Verlag, New York, 1992.

\bibitem{barker}
L.~Barker.
\newblock Rhetorical biset functors, rational $p$-biset functors and their
  semisimplicity in characteristic zero.
\newblock {\em Journal of Algebra}, 319:3810--3853, 2008.

\bibitem{biset}
S.~Bouc.
\newblock {\em Biset functors for finite groups}.
\newblock Springer, Berlin, 2010.

\bibitem{centros}
S.~Bouc and N.~Romero.
\newblock The center of a {G}reen biset functors.
\newblock {\em Pacific Journal of Mathematics}, 303:459--490, 2019.

\bibitem{curtis}
C.~Curtis, {I.}~Reiner.
\newblock {\em Methods of representation theory with applications to finite
  groups and orders}.
\newblock Volume 1. John Wiley and sons, U.S.A, 1981.

\bibitem{tesbenja}
B.~Garc\'ia.
\newblock {\em Construcci\'on de Yoneda-Dress de funtores de representaciones
  lineales}.
\newblock Tesis de Doctorado UMSNH-UNAM, Morelia, 2018.

\bibitem{garcia-shifted}
B.~Garc\'{\i}a.
\newblock On the ideals and essential algebras of shifted functors of linear
  representations.
\newblock {\em J. Algebra}, 521:452--480, 2019.

\bibitem{lachica}
N.~Romero.
\newblock Simple modules over {G}reen biset functors.
\newblock {\em Journal of Algebra}, 367:203--221, 2012.

\end{thebibliography}

\centerline{\rule{5ex}{.1ex}}
\begin{flushleft}
Serge Bouc, CNRS-LAMFA, Universit\'e de Picardie, 33 rue St Leu, 80039, Amiens, France.\\
{\tt serge.bouc@u-picardie.fr}\vspace{1ex}\\
Nadia Romero, DEMAT, UGTO, Jalisco s/n, Mineral de Valenciana, 36240, Guanajuato, Gto., Mexico.\\
{\tt nadia.romero@ugto.mx}
\end{flushleft}

\end{document}